\documentclass[11pt,letterpaper]{article}

\usepackage{float}

\usepackage{amscd}
\usepackage{mathrsfs}
\usepackage[IL2]{fontenc}

\usepackage{comment} 

\usepackage{color}
\usepackage{amssymb,amsmath,amsthm,amsfonts,latexsym}
\usepackage[utf8x]{inputenc}
\usepackage{bbm} 

\usepackage{graphicx}
\usepackage[textsize=tiny]{todonotes}

\usepackage{amsmath,amsfonts,amssymb,amsthm} 
\usepackage{tikz}
\usepackage{hyperref}
\usepackage{thmtools}
\usepackage{thm-restate}

\newtheorem{theorem}{Theorem}[section]

\newtheorem*{claim*}{Claim}
\newtheorem*{theorem*}{Theorem}
\newtheorem*{definition*}{Definition}
\newtheorem*{remark*}{Remark}

\newtheorem{corollary}[theorem]{Corollary}
\newtheorem{lemma}[theorem]{Lemma}
\newtheorem{remark}[theorem]{Remark}
\newtheorem{claim}[theorem]{Claim}

\newtheorem{proposition}[theorem]{Proposition}

\newtheorem{definition}[theorem]{Definition}
\newtheorem{question}[theorem]{Question}

\newcommand{\comm}[1]{}
\usepackage[capitalize, nameinlink]{cleveref}
\crefname{theorem}{Theorem}{Theorems}
\crefname{proposition}{Proposition}{Propositions}
\crefname{observation}{Observation}{Observations}
\crefname{lemma}{Lemma}{Lemmas}
\crefname{claim}{Claim}{Claims}
\crefname{problem}{Problem}{Problems}
\crefname{conjecture}{Conjecture}{Conjectures}
\crefname{question}{Question}{Questions}
\crefname{example}{Example}{Examples}
\crefname{fact}{Fact}{Facts}

\newcommand{\local}{\mathsf{LOCAL}}

\newcommand{\Cay}{\operatorname{Cay}}

\newcommand{\fD}{\mathcal{D}}

\newcommand{\fG}{\mathcal{G}}
\newcommand{\fH}{\mathcal{H}}

\newcommand{\N}{\mathbb{N}}

\newcommand{\PI}{\operatorname{Alice}}
\newcommand{\PK}{\operatorname{Bob}}

\def\leukfrac#1/#2{\leavevmode
               \kern.1em
                \raise.9ex\hbox{\the\scriptfont0 ${}_#1$}
                \hskip -1pt\kern-.1em
                /\kern-.15em\lower.10ex\hbox{\the\scriptfont0 ${}_#2$}}

\makeatletter
\def\diam{\mathop{\operator@font diam}\nolimits}
\makeatother

\newcommand{\proj}{proj}
\newcommand{\schreier}[3]{Sch(#1,#2,#3)}
\newcommand{\bbg}{\mathbf{\Gamma}}
\newcommand{\bbo}{\mathbf{\Delta}^1_1}
\newcommand{\bp}{\mathbf{\Pi}^1_1}
\newcommand{\oom}{\mathbb{N}^\mathbb{N}}

\newcommand{\bs}{\mathbf{\Sigma}^1_1}

\newcommand{\om}{\mathbb{N}}

\newcommand{\omm}{[\mathbb{N}]^\mathbb{N}}

\newcommand{\concatt}{%
	\mathbin{\raisebox{1ex}{\scalebox{.7}{$\frown$}}}%
}

\newcommand{\mc}{\mathcal}
\newcommand{\mb}{\mathbf}
\newcommand{\HOM}{\operatorname{\bf Hom}}
\newcommand{\AD}{\mathtt{AD}}
\newcommand{\DC}{\mathtt{DC}_{\aleph_0}}
\newcommand{\PD}{\mathtt{PD}}
\newcommand{\Homac}{\HOM^{ac}}
\newcommand{\Homed}{\HOM^{e}}

\newcommand{\Deltachrom}{\chi_{\Game \mathbf{\Delta}^1_1}}
\newcommand{\Edgelabel}[1]{el\chi(#1)}
\newcommand{\Edgelabeldelta}[1]{{el\chi}_{\Game \mathbf{\Delta}^1_1}(#1)}

\newcommand{\Root}{Root}

\newtheorem{innercustomthm}{Theorem}
\newenvironment{customthm}[1]
{\renewcommand\theinnercustomthm{#1}\innercustomthm}
{\endinnercustomthm}

\title{\vspace{-1cm} On Homomorphism Graphs}

\usetikzlibrary{math}
\usetikzlibrary{arrows.meta}
\usetikzlibrary{shapes,decorations}
\usetikzlibrary{decorations.pathmorphing}
\usetikzlibrary{calc}
\definecolor{pastelred}{rgb}{1.0, 0.41, 0.38}
\definecolor{pastelblue}{rgb}{0.52, 0.63, 0.94}
\definecolor{pastelyellow}{rgb}{0.99, 0.99, 0.59}
\definecolor{pastelgreen}{rgb}{0.47, 0.87, 0.47}
\definecolor{pastelorange}{rgb}{1.0, 0.7, 0.28}
\begin{document}
	
	\newcommand*\samethanks[1][\value{footnote}]{\footnotemark[#1]}
	
	\author{
		Sebastian Brandt \\
		\small CISPA Helmholtz Center for Information Security \\
		\small \texttt{brandt@cispa.de}\\
		\and 
		Yi-Jun Chang\thanks{Supported by  Dr.~Max R\"{o}ssler, by the Walter Haefner Foundation, and by the ETH Z\"{u}rich Foundation.} \\
		\small National University of Singapore \\
		\small \texttt{cyijun@nus.edu.sg} \\
		\and 
		Jan Greb\'{i}k\thanks{Supported by Leverhulme Research Project Grant RPG-2018-424. Main part of this work was carried out while affiliated with \emph{University of Warwick.}} 
		\\
		\small Masaryk University and UCLA\\
		\small \texttt{grebikj@math.ucla.edu}\\
		\and 
		Christoph Grunau\thanks{Supported by the European Research Council (ERC) under the European
			Unions Horizon 2020 research and innovation programme (grant agreement No.~853109). 
		}\\
		\small ETH Z\"{u}rich \\
		\small \texttt{cgrunau@inf.ethz.ch}\\
		\and 
		V\'{a}clav Rozho\v{n}\samethanks \\
		\small ETH Z\"{u}rich \\
		\small \texttt{rozhonv@ethz.ch}\\
		\and 
		Zolt\'{a}n Vidny\'{a}nszky\thanks{Partially supported by the
			National Research, Development and Innovation Office
			-- NKFIH, grants no.~113047, no.~129211 and FWF Grant M2779.  Main part of this work was carried out while affiliated with \emph{Caltech}.}\\
		\small E\"otv\"os University, Budapest\\
		\small \texttt{zoltan.vidnyanszky@ttk.elte.hu}\\}
	
	\insert\footins{\footnotesize{MSC codes: Primary 03E15, 05C15; Secondary 28A05}}
	\insert\footins{\footnotesize{Key Words: $\Sigma^1_2$-complete, Borel chromatic number, bounded degree, Borel graph, Brooks' theorem}}

	\date{}
	
	\maketitle 
	\pagenumbering{gobble}   
	
	\vspace{-1cm}
	\begin{abstract}
		
		We introduce a new type of examples of bounded degree acyclic Borel graphs and study their combinatorial properties in the context of descriptive combinatorics, using a generalization of the determinacy method of Marks \cite{DetMarks}.
		The motivation for the construction comes from the adaptation of this method to the $\local$ model of distributed computing \cite{brandt_chang_grebik_grunau_rozhon_vidnyaszky2021LCLs_on_trees_descriptive}. Our approach unifies the previous results in the area, as well as produces new ones.
		In particular, strengthening the main result of \cite{todorvcevic2021complexity}, we show that for $\Delta>2$ it is impossible to give a simple characterization of acyclic $\Delta$-regular Borel graphs with Borel chromatic number at most $\Delta$: such graphs form a $\mathbf{\Sigma}^1_2$-complete set.
		This implies a strong failure of Brooks'-like theorems in the Borel context. 
	\end{abstract}

	\pagenumbering{arabic}

	\section{Introduction}

	\emph{Descriptive combinatorics} is an area concerned with the investigation of combinatorial problems on infinite graphs that satisfy additional regularity properties (see, e.g., \cite{pikhurko2021descriptive_comb_survey,kechris_marks2016descriptive_comb_survey} for surveys of the most important results). In recent years, the study of such problems revealed a deep connection to other areas of mathematics and computer science.
	The most relevant to our study are the connections with the so-called $\local$ model from the area of distributed computing.
	There are several recent results that use distributed computing techniques in order to get results either in descriptive combinatorics \cite{Bernshteyn2021LLL,Bernshteyn2021local=cont,brandt_chang_grebik_grunau_rozhon_vidnyaszky2021LCLs_on_trees_descriptive,elek2018qualitative,grebik_rozhon2021LCL_on_paths}, or in the theory of random processes \cite{HolroydSchrammWilson2017FinitaryColoring,grebik_rozhon2021toasts_and_tails}. 
	
	The starting point of our work was the investigation of the opposite direction. Namely, our aim was to adapt the celebrated \emph{determinacy technique} of Marks \cite{DetMarks} to the $\local$ model of distributed computing. In order to perform the adaptation (which is indeed possible, see our conference paper \cite{brandt_chang_grebik_grunau_rozhon_vidnyaszky2021LCLs_on_trees_descriptive}\footnote{The connection between the current paper and the conference paper \cite{brandt_chang_grebik_grunau_rozhon_vidnyaszky2021LCLs_on_trees_descriptive} is the following.
		The latter paper builds a theory of local problems on trees from several perspectives and aims to a broader audience.
		The original version of this paper should have been a journal version of some results from \cite{brandt_chang_grebik_grunau_rozhon_vidnyaszky2021LCLs_on_trees_descriptive} aiming to people working in descriptive combinatorics.
		In the end, we added several new applications of our method that cannot be found in \cite{brandt_chang_grebik_grunau_rozhon_vidnyaszky2021LCLs_on_trees_descriptive}.}), we had to circumvent several technical hurdles that, rather surprisingly, lead to the main objects that we study in this paper, \emph{homomorphism graphs} (defined in \cref{s:hom}). 
	We refer the reader to \cite{brandt_chang_grebik_grunau_rozhon_vidnyaszky2021LCLs_on_trees_descriptive} for a detailed discussion of the concepts and their connections to the $\local$ model.

	Before we state our results, we recall several basic notions and facts.
	A \emph{graph} $G$ on a set $X$ is a symmetric subset of $X^2 \setminus \{(x,x):x \in X\}$. We will refer to $X$ as the \emph{vertex set of $G$}, in symbols $V(G)$, and to $G$ as the \emph{edge set}. If $n \in \{1,2,\dots,\aleph_0\}$, a \emph{(proper) $n$-coloring of $G$} is a mapping $c:V(G) \to n$ such that $(x,y) \in G \implies c(x) \neq c(y)$. The chromatic number of $G$, $\chi(G)$ is the minimal $n$ for which an $n$-coloring exists. If $G$ and $H$ are graphs, a \emph{homomorphism} from $G$ to $H$ is a mapping $c:V(G) \to V(H)$ that preserves edges. Note that $\chi(G) \leq n$ if and only if $G$ admits a homomorphism to the complete graph on $n$ vertices, $K_n$.
	We denote by $\Delta(G)$ the supremum of the vertex degrees of $G$. 
	In what follows, we will only consider graphs with degrees bounded by a finite number, unless explicitly stated otherwise. A graph is called \emph{$\Delta$-regular} if every vertex has degree $\Delta$. It is easy to see that $\chi(G) \leq \Delta(G)+1$.
	Moreover, Brooks' theorem states that this inequality is sharp only in trivial situations: if $\Delta(G)>2$, it happens if and only if $G$ contains a complete graph on $\Delta(G)+1$ vertices, and if $\Delta(G)=2$, it happens if and only if $G$ contains an odd cycle.
	
	We say that $G$ is a \emph{Borel} graph if $V(G)$ is a standard Borel space, see \cite{kechrisclassical}, and the set of edges of $G$ is a Borel subset of $V(G)\times V(G)$ endowed with the product Borel structure.
	The \emph{Borel chromatic number, $\chi_B(G)$}, of $G$ is defined as the minimal $n$ for which a Borel $n$-coloring exists, here we endow $n$ with the trivial Borel structure.
	Similar concepts are studied when we relax the notion of Borel measurable to merely \emph{measurable} with respect to some probability measure, or \emph{Baire} measurable with respect to some compatible Polish topology.
	
	It has been shown by Kechris-Solecki-Todor\v{c}evi\'c \cite{KST} that $\chi_B(G) \leq \Delta(G)+1$, and it was a long standing open problem, whether Brooks' theorem has a literal extension to the Borel context, at least in the case $\Delta(G)>2$. For example, it has been proved by Conley-Marks-Tucker-Drob \cite{conleymarks} that in the measurable or Baire measurable setting the answer is affirmative. Eventually, this problem has been solved by Marks 
	\cite{DetMarks}, who showed the existence of $\Delta$-regular acyclic Borel graphs with Borel chromatic number $\Delta+1$. Remarkably, this result relies on Martin's Borel determinacy theorem, one of the cornerstones of modern descriptive set theory.
	
	\subsection*{Results}
	
	First let us give a high-level overview of our new method for constructing Borel graphs (for the precise definition of the notions discussed below see \cref{s:hom}). Fix a $\Delta>2$. To a given Borel graph $\mathcal{H}$ we will associate a $\Delta$-regular acyclic Borel graph $\Homac(T_\Delta,\mathcal{H})$. Roughly speaking, the vertex set of the graph will be a collection of pairs $(x,h)$, where $x \in V(\mathcal{H})$ and $h$ is a homomorphism from the $\Delta$-regular infinite rooted tree $T_\Delta$ to $\mathcal{H}$ that maps the root to $x$, and $(x,h)$ is adjacent to $(x',h')$ if $h'$ is obtained from $h$ by moving the root to a neighbouring vertex.
	
	The main idea is that the we can use the combinatorial properties of $\mathcal{H}$ to control the properties of $\Homac(T_\Delta,\mathcal{H})$. Most importantly, we will argue that from a Borel $\Delta$-coloring $c$ of $\Homac(T_\Delta,\mathcal{H})$ we can construct a $\Delta$-coloring of $\mathcal{H}$: to each $x$ we associate games analogous to the ones developed by Marks, in order to select one of the sets $\{h:c(x,h)=i\}$ for $i \leq \Delta$  (in some sense, we select the largest), and color $x$ with the appropriate $i$. As this selection will be based on the existence of winning strategies, the coloring of $\mathcal{H}$ will not be Borel. However, it will still be in a class that has all the usual regularity properties (see \cref{s:prel} for the definition of this class and the corresponding chromatic number, $\Deltachrom$). Thus we will be able to prove the following.
	
	\begin{theorem}
		\label{t:implication}
		Let $\mathcal{H}$ be a locally countable Borel graph.
		Then we have 
		$$\Deltachrom(\mathcal{H})>\Delta \ \Rightarrow \ \chi_B(\Homac(T_\Delta,\mathcal{H}))>\Delta.$$
		In particular,  $\Deltachrom(\mathcal{H})>\Delta$ holds if the Ramsey measurable (if $V(\fH)=[\N]^\N$), Baire measurable, or $\mu$ measurable chromatic number of $\mathcal{H}$ is $>\Delta$.
	\end{theorem}
	
	Next we list the applications.
	In each instance we use a version of \cref{t:implication} for a carefully chosen target graph $\mathcal{H}$.
	These graphs come from well-studied contexts of descriptive combinatorics, namely, \emph{Ramsey property} and \emph{Baire category}.

	\paragraph{a) Complexity result.}
	We apply homomorphism graphs in connection to projective complexity and Brooks' theorem. One might conjecture that the right generalization of Brooks' theorem to the Borel context is that Marks' examples serve as the analogues of the complete graph, i.e., whenever $G$ is a Borel graph with $\chi_B(G)=\Delta(G)+1$, then $G$ must contain a Borel homomorphic copy of the corresponding example of Marks. Note that in the case $\Delta(G)=2$ this is the situation, as there is a Borel analogue of odd cycles that admits a homomorphism into each Borel graph $G$ with $\chi_B(G)>2$ (see \cite{benen}). 
	
	In \cite{todorvcevic2021complexity} it has been shown that it is impossible to give a simple characterization of acyclic Borel graphs with Borel chromatic number $\leq 3$. The construction there was based on a Ramsey theoretic statement, the Galvin-Prikry theorem \cite{galvin1973borel}. An important weakness of that proof is that it uses graphs of finite but unbounded degrees. Using the homomorphism graph combined with the method developed in \cite{todorvcevic2021complexity} and Marks technique, we obtain the analogous result for bounded degree graphs. 
	
	\begin{theorem}
		\label{t:mainc} For each $\Delta>2$ the family of $\Delta$-regular acyclic Borel graphs with Borel chromatic number $\leq \Delta$ has no simple characterization, namely, it is $\mathbf{\Sigma}^1_2$-complete.
	\end{theorem}
	From this we deduce a strong negative answer to the conjecture described above. 
	\begin{corollary}
		Brooks' theorem has no analogue for Borel graphs in the following sense.
		Let $\Delta>2$.
		There is no countable family $\{\mathcal{H}_i\}_{i \in I}$ of Borel graphs such that for any  Borel graph $\mathcal{G}$ with $\Delta(\mathcal{G})\leq \Delta$ we have $\chi_B(\mathcal{G})>\Delta$ if and only if for some $i \in I$ the graph $\mathcal{G}$ contains a Borel homomorphic copy of $\mathcal{H}_i$. 
	\end{corollary}

	\paragraph{b) Chromatic number and hyperfiniteness.}
	Recall that a Borel graph $\mathcal{G}$ is called \emph{hyperfinite}, if it is the increasing union of Borel graphs with finite connected components.
	In \cite{conleyhyp} the authors examine the connection between \emph{hyperfiniteness} and notions of Borel combinatorics, such as Borel chromatic number and the Lov\'asz Local Lemma.
	Roughly speaking, they show that hyperfiniteness has no effect on Borel combinatorics, for example, they establish the following. 
	
	\begin{theorem}[\cite{conleyhyp}] \label{t:hyperfiniteness}
		There exists a hyperfinite $\Delta$-regular acyclic Borel graph with Borel chromatic number $\Delta+1$.  
	\end{theorem}
	
	Using homomorphism graphs, we provide a new, short and more streamlined proof of this result.
	In particular, the conclusion about the chromatic number follows from our general result about $\Homed$ (a version of $\Homac$), while to get hyperfiniteness we can basically choose any acyclic hyperfinite graph as a target graph.
	To get both properties at once, we simply pick a variant of the graph $\mathbb{G}_0$ (see \cite[Section 6]{KST}) as our target graph.
	
	\paragraph{c) Graph homomorphism.}
	We also consider a slightly more general context: homomorphisms to finite graphs. Clearly, the $\Delta$-regular examples constructed by Marks do not admit a Borel homomorphism to finite graphs of chromatic number at most $\Delta$, as this would imply that their Borel chromatic number is $\leq \Delta$.
	No other examples of such graphs were known.
	We show the following.

	\begin{restatable}{theorem}{MainBorel}\label{thm:MainBorel}
		For every $\Delta>2$ and every $\ell \leq 2\Delta-2$ there are a finite graph $H$ and a $\Delta$-regular acyclic Borel graph $\fG$ such that $\chi(H)=\ell$ and $\fG$ does not admit Borel homomorphism to $H$. The graph $\mathcal{G}$ can be chosen to be hyperfinite. 
	\end{restatable}
	
	This theorem is a step towards the better understanding of Problem 8.12  from \cite{kechris_marks2016descriptive_comb_survey}.
	
	\begin{remark}
		The upper bound $2\Delta-2$ on the chromatic number is implied by the combinatorial condition \emph{almost $\Delta$-colorable}, see \cref{def:DeltaStar}, that is utilized in the generalization of Marks' determinacy technique.
		It is an interesting open problem to determine exactly to what graphs the determinacy argument may be applied.
		
		Recently Cs\'oka and the last author showed that there is no \emph{factor of iid} homomorphism from the $\Delta$-regular tree to finite graphs of arbitrarily large chromatic number using the theory of entropy inequalities.
		This, of course, implies the same result in the Borel setting.
		Observe, however, that there is a factor of iid homomorphism from the $\Delta$-regular tree to examples constructed in this paper as the universal graph $H_\Delta$ that is almost $\Delta$-colorable, see \cref{subsec:measure} and \cref{fig:h3}, contains the complete graph on $\Delta$ vertices.
		This shows that the difference between factor of iid and Borel chromatic numbers extends non-trivially to the question about graph homomorphism.
		The exact relationship between the existence of factor of iid and Borel homomorphisms is wide open.
	\end{remark}

	\paragraph{Roadmap.}

	The paper is structured as follows. In \cref{s:prel} we collect the most important definitions and theorems that are going to be used. Then, in \cref{s:hom} we establish the basic properties of homomorphism graphs and their various modifications. \cref{s:marks} contains Marks' technique's adaptation to our context, while in \cref{s:applications} we prove our main results. We conclude the paper with a couple of remarks in \cref{s:problems}.

	\section{Preliminaries}
	\label{s:prel}
	For standard facts and notations of descriptive set theory not explained here we refer the reader to \cite{kechrisclassical} (see also \cite{moschovakis2009descriptive}).
	
	Given a graph $G$, we refer to maps $V(G) \to S$ and $G \to S$ as \emph{vertex ($S$)-labelings} and \emph{edge ($S$)-labelings}, respectively. An edge labeling is called an \emph{edge coloring}, if incident edges have different labels.
	Let $\mathcal{F}$ be a family of subsets of $V(G)$, and $n \in \{1,2,\dots, \aleph_0\}$. An \emph{$\mathcal{F}$ measurable $n$-coloring} is an $n$-coloring $c$ of $G$ such that $c^{-1}(i) \in \mathcal{F}$ for each $i<n$. Using this notion, we define the \emph{$\mathcal{F}$ measurable chromatic number of $G$}, $\chi_\mathcal{F}(G)$ to be the minimal $n$ for which such a coloring exists. 
	
	We denote by $[S]^\N$ the collection of infinite subsets of the set $S$, and by $S^{<\N}$ the family of finite sequences of elements of $S$. Points of the space $[\N]^\N$ will be identified with their increasing enumeration, making $[\N]^\N$ a $G_\delta$ subset of $
 \N^\N$, and hence the product topology of $\N^\N$ gives rise to a Polish topology on $[\N]^\N$.
	
	Define the \emph{shift-graph (on $[\mathbb{N}]^\mathbb{N}$)}, $\mathcal{G}_S$, by letting $x$ and $y$ be adjacent if $y=x \setminus \min x$ or $x=y \setminus \min y$. The shift-graph has a close connection to the notion of so called Ramsey property: for $s \subset \mathbb{N}$ finite and $A \in [\mathbb{N}]^\mathbb{N}$ with $\max s < \min A$ let $[s,A]=\{B \in [\mathbb{N}]^\mathbb{N}: s \subset B, A \supseteq B \setminus s\}$. A set $S \subseteq [\mathbb{N}]^\mathbb{N}$ is called \emph{Ramsey} if for each set of the form $[s,A]$ there exists $B\in [A]^\N$ such that $[s,B] \cap S= \emptyset$ or $[s,B] \subseteq S$ (see, e.g., \cite{KST,khomskii2012regularity,ramseyspaces} for results on the shift-graph and Ramsey measurability). The following follows from the definition.
	
	\begin{theorem}
		\label{t:galvinprikry}
		The graph $\mathcal{G}_S$ has no Ramsey measurable finite coloring. 
	\end{theorem}

	Note that the Galvin-Prikry theorem asserts that Borel sets are Ramsey measurable. However, adapting Marks' technique to our setting will require the usage of families of sets that are much larger than the collection of Borel sets. 
	
	If $T \subseteq \N^{<\N}$ is a nonempty pruned tree, and $A \subseteq \N^\N$, $G(T,A)$ will denote the two-player infinite game on $\N$ with legal positions in $T$ and payoff set $A$. We will call the first player $\PI$ and the second $\PK$, $\PI$ wins the game if the resulting element is in $A$. Note that the Borel Determinacy Theorem \cite{martin} states that one of the players has a winning strategy in $G(T,A),$ whenever $A$ is Borel. 
	
	Recall that a subset of $A$ a Polish space $X$ is in the class $\Game \mathbf{\Delta}^1_1$ if there is some Borel set $B \subset X \times \N^\N$ such that
	$$A=\{x:\PI \text{ has a winning strategy in } G(\N^{<\N},B_x)\},$$
	see \cite{moschovakis2009descriptive}. 
	
	By modifying the payoff sets, it is easy to see the following (see, \cite[p138]{kechrisclassical}).
	\begin{lemma}
		\label{l:gamesection}
		Let $X$ be a Polish space and $B \subseteq X \times \N^\mathbb{N}$ be Borel and $x \mapsto T_x$ be a Borel map such that $T_x$ is a pruned subtree of $\N^{<\N}$. Then the set \[\{x \in X: \text{ $\PI$ has a winning strategy in } G(T_x,B_x)\}\] is $\Game \mathbf{\Delta}^1_1$.
	\end{lemma}
	
	The class $\Game \mathbf{\Delta}^1_1$ enjoys a number of regularity properties.
	\begin{proposition} 
		\label{f:provably}
		Let $X$ be a Polish space. $\Game \mathbf{\Delta}^1_1$ sets
		\begin{enumerate}
			\item form a $\sigma$-algebra,
			\item \label{c:Baireprop} have the Baire property w.r.t. any compatible Polish topology,
			\item \label{c:univmeas} are measurable w.r.t. any Borel probability measure,
			\item \label{c:Ramseyprop} in the case $X=[\mathbb{N}]^\mathbb{N}$ have the Ramsey property.
		\end{enumerate}
	\end{proposition}
	
	Before proving this statement we need to fix an encoding of Borel sets. Let $\mb{BC}(X)$ be a set of Borel codes and sets $\mb{A}(X)$ and $\mb{C}(X)$ with the properties summarized below:

	\begin{proposition} (see \cite[3.H]{moschovakis2009descriptive})
		\label{f:prel}
		\begin{itemize}
			\item $\mb{BC}(X) \in \bp(\oom)$, $\mb{A}(X) \in \bs(\oom \times X)$, $\mb{C}(X) \in \bp(\oom \times X)$,
			\item for $c\in \mb{BC}(X)$ and $x \in X$ we have $(c,x) \in \mb{A}(X) \iff (c,x) \in \mb{C}(X)$,
			\item if $P$ is a Polish space and $B \in \bbo(P \times X)$ then there exists a Borel map $f:P \to \oom$ so that $ran(f) \subset \mb{BC}(X)$ and for every $p \in P$ we have $\mathbf{A}(X)_{f(p)}=B_p$.
		\end{itemize}
		Moreover, in the case $X=(\N^\N)^k$ the sets $\mb{BC}(X)$, $\mb{A}(X)$, and $\mb{C}(X)$ can be described by $\Pi^1_1$, $\Sigma^1_1$, and $\Pi^1_1$ formulas, respectively. 
	\end{proposition} 
	
	Now we can prove \cref{f:provably}.
	\begin{proof}[Proof of \cref{f:provably}]
		The first statement is a consequence of \cite[Lemma 6D.1]{moschovakis2009descriptive} and not very hard too verify.  
		
		In order to see the rest, we rely on results of Feng-Magidor-Woodin \cite{feng1992universally}. Recall that a subset $S$ of $\N^\N$ is \emph{universally Baire} if for every topological space $Y$ that has a basis consisting of regular open sets and every continuous $f:Y \to \N^\N$ the set $f^{-1}(S)$ has the Baire property. 
		
		It is not hard to check the following properties of universally Baire sets (see \cite[Theorem 2.2]{feng1992universally} and the discussion on p212).
		
		\begin{claim}
			\label{cl:propuniv}
			Assume that $S \subseteq \N^\N$ is universally Baire, $X$ is a Polish space and $\varphi:X \to \N^\N$ be an injection.
			\begin{enumerate}
				\item $S$ has the Baire property w.r.t. any compatible Polish topology, is measurable w.r.t. any Borel probability measure and, in case $S \subseteq [\N]^\N$, $S$ has the Ramsey property.
				\item If $X \subseteq \N^\N$ and $\varphi$ is continuous then $\varphi^{-1}(S)$ is universally Baire. 
				\item If $\varphi$ is Borel then $\varphi^{-1}(S)$ has the Baire property w.r.t. any compatible Polish topology, is measurable w.r.t. any Borel probability measure.
			\end{enumerate}
		\end{claim}

		Now we consider a set which encompasses all the $\Game \mathbf{\Delta}^1_1$ sets, and show that it is universally Baire. To ease the notation set $\mb{BC}:=\mb{BC}(\N^\N \times \N^\N)$, $\mb{C}:=\mb{C}(\N^\N \times \N^\N)$ and $\mb{A}:=\mb{A}(\N^\N \times \N^\N)$.
		Note that we have $\mb{BC} \in \bp(\oom)$, $\mb{A} \in \bs(\oom \times (\oom\times \oom))$, $\mb{C} \in \bp(\oom \times (\oom\times\oom))$.
		Consider first the set \[WS=\{(x,c):x \in \N^\N, c \in \mb{BC},\] \[\text{ $\PI$ has a winning strategy in $G(\N^{<\N},(\mb{C}_c)_x)$}\}.\]
		
		\begin{claim}
			\label{cl:WS} $WS$ is universally Baire.
		\end{claim}
		\begin{proof}
			
			By the characterization result of Feng-Magidor-Woodin, a set $Z \subseteq \N^\N$ is universally Baire if and only if there exist a cardinal $\lambda$ and trees $T,T^* \subseteq \N^{<\N} \times \lambda^{<\N}$ such that $Z=proj_{\N^\N}([T])$ and $\N^\N \setminus Z = proj_{\N^\N}([T^*])$ and for every forcing notion $\mathbb{P}$ we have that
            $$\mathbb{P} \Vdash \mathbb{N}^\mathbb{N}=proj_{\N^\N}([T]) 
			\sqcup proj_{\N^\N}([T^*]).$$
			
			We show this condition for $WS$. Note that $S$ is a winning strategy for $\PI$ in $G(\N^{<\N},(\mb{C}_c)_x)$ if and only if $S$ is a strategy for $\PI$ such that $\forall y \ (y \not \in [S] \lor y \in (\mb{C}_c)_x)$. Using the last sentence of \cref{f:prel} this yields that $WS$ can be described by a $\Sigma^1_2$ formula. Moreover, as $\mb{A}_c=\mb{C}_c$ whenever $c \in \mb{BC}$, and $\mb{C}_c$ is Borel, by the Borel determinacy theorem $\PI$ has a winning strategy in $G(\N^{<\N},(\mb{C}_c)_x)$ if and only if $\PK$ has no winning strategy. That is, for every $S$ strategy for $\PK$ we have $\exists y \ (y \in [S] \land y \in (\mb{A}_c)_x)$. This yields a description of $\N^\N \setminus WS$ using a $\Sigma^1_2$ formula. By \cite[2D.3]{moschovakis2009descriptive} there are trees $T,T^* \subseteq \N^{<\N} \times \omega_1^{<\N}$ (together with formulas defining them) such that $WS=proj_{\N^\N}([T])$ and $\N^\N \setminus WS = proj_{\N^\N}([T^*])$ holds in every model of ZFC, showing that the desired property holds.			
		\end{proof}
	
		\begin{claim}
			\label{cl:triv}
			Let $S$ be a $\Game \mathbf{\Delta}^1_1$ subset of a Polish space $X$. Then there is a Borel injection $\varphi$ and some $c \in \N^\N$ such that $S=\varphi^{-1}(WS^{c})$, where $WS^c=\{x\in \mathbb{N}^\mathbb{N}:(x,c)\in WS\}$. If $X$ is zero dimensional then there is even a homeomorphism with such a property. 
		\end{claim}
		\begin{proof}
			Let $B$ be a Borel set in $X \times \N^\N$ such that 
			\[S=\{x \in X: \text{Alice has a winning strategy in $G(\N^{<\N},B_x)$}\}.\]
			Since $X$ is Polish, there exists a Borel injection $\varphi:X \to \N^\N$. Moreover, if $X$ is zero-dimensional, then $\varphi$ can be chosen to be a homeomorphism. Now let $(s,r) \in B' \iff (\varphi^{-1}(s),r) \in B$. Then $B'$ is Borel, and there is some $c \in \mathbf{BC}$ with $\mathbf{C}_c=B'$. Then, by definition we have $S=\varphi^{-1}(WS^{c})$. 
		\end{proof}

		Now let first $S$ be a $\Game\mathbf{\Delta}^1_1$ subset of an arbitrary Polish space $X$. By the claim above, there exist a Borel injection $\varphi$ and a $c$ with $\varphi^{-1}(WS^{c})$. 
		Using \cref{cl:propuniv} $S$ has properties \eqref{c:Baireprop} and \eqref{c:univmeas}. 
		
		Finally, if $X=[\N]^\N$ then $\varphi$ can be taken to be a homeomorphism. Then $S$ is universally Baire and hence Ramsey measurable by \cref{cl:propuniv}.

	\end{proof}

	\begin{remark}
		In the original version of this paper we introduced the family of weakly provably $\mathbf{\Delta}^1_2$ sets, which contains $\Game \mathbf{\Delta}^1_1$ sets, in order to be able to handle definability issues. Subsequently Kastner and Lyons pointed out that these results also follow from the theorems of Feng-Magidor-Woodin. 
		
		In an upcoming work Kastner and Lyons will prove the required regularity properties of the class $\Game \mathbf{\Delta}^1_1$ with a more streamlined, purely game theoretic argument, based on the one given in Kechris \cite{kechris1978forcing}.  

        Finally, all the regularity properties of $\Game \mathbf{\Delta}^1_1$ sets follow from projective determinacy.
		
	\end{remark}

	\section{The homomorphism graph}

	\label{s:hom}
	
	In this section we define the main objects of our study, \emph{homomorphism graphs}, and establish a couple of their properties. 
	
	Let $\Gamma$ be a countable group and $S \subseteq \Gamma$ be a generating set. Assume that $\Gamma\curvearrowright X$ is an action of $\Gamma$ on the set $X$.
	As there is no danger of  confusion we always denote the action with the symbol~$\cdot$.
	The \emph{Schreier graph} $\schreier{\Gamma}{S}{X}$ of such an action is a graph on the set $X$ such that $x \neq x'$ are adjacent iff for some $\gamma \in S \cup S^{-1}$ we have that $\gamma \cdot x=x'$. 
	
	Probably the most important example of a Schreier graph is the (right) \emph{Cayley graph, $\Cay(\Gamma,S)$} that comes from the right multiplication action of $\Gamma$ on itself.
	That is, $g,h\in \Gamma$ form an edge in $\Cay(\Gamma,S)$ if there is $\sigma\in S$ such that $g\cdot \sigma=h$.
	Another example is the graph of the left-shift action of $\Gamma$ on the space $2^\Gamma$: recall that the left-shift action is defined by \[\gamma \cdot x(\delta)=x(\gamma^{-1} \cdot \delta)\]
	for $\gamma \in \Gamma$ and $x \in 2^\Gamma$. 
	Observe that the Schreier graph of this actions is a Borel graph, when we endow the space $2^\Gamma$ with the product topology. 
	
	Our examples will come from a generalization of this graph. First note that if we replace $2$ by any other standard Borel space $X$, the space $X^\Gamma$ still admits a Borel product structure with respect to which the Schreier graph of the left-shift action defined as above, is a Borel graph. 
	The main idea is to start with a Borel graph $\mathcal{H}$ and restrict the corresponding Schreier graph on $V(\mathcal{H})^\Gamma$ to an appropriate subset on which the elements $h \in V(\mathcal{H})^\Gamma$ are graph homomorphisms from $\Cay(\Gamma,S)$ to $\mathcal{H}$. This allows us to control certain properties (such as chromatic number or hyperfiniteness) of the resulting graph by the properties of $\mathcal{H}$. More precisely:
	
	\begin{definition}
		Let $\mathcal{H}$ be a Borel graph and $\Gamma$ be a countable group with a generating set $S$. Let $\HOM(\Gamma,S,\fH)$ be the restriction of $\schreier{\Gamma}{S}{V(\mathcal{H})^\Gamma}$ to the set
		\[\{h \in V(\mathcal{H})^\Gamma: \text{$h$ is a graph homomorphism from $\Cay(\Gamma,S)$ to $V(\mathcal{H}$)}\}.\]
	\end{definition}
	
	We will refer to $\mathcal{H}$ as the \emph{target graph}, and we will denote the map $h \mapsto h(1)$ by $\Root$ (note that the vertices of $\Cay(\Gamma,S)$ are labeled by the elements of $\Gamma$). It is clear from the definition that $\HOM(\Gamma,S,\fH)$ is a Borel graph with degrees at most $|S \cup S^{-1}|$ and that $\Root$ is a Borel map. We can immediately make the following observation.
	
	\begin{proposition}
		\label{pr:chromatic}
		The map $\Root$ is a Borel homomorphism from $\HOM(\Gamma,S,\fH)$ to $\fH$.
		
		In particular, $\chi_B(\HOM(\Gamma,S,\fH)) \leq \chi_B(\mathcal{H})$ and the action of $1 \neq \gamma \in S$ on $\HOM(\Gamma,S,\fH)$ has no fixed-points.
	\end{proposition} 
	\begin{proof}
		Let $h\in \HOM(\Gamma,S,\fH)$ and $1\not= \gamma\in S\cup S^{-1}$.
		Note that as $\gamma^{-1}$ and $1$ are adjacent in $\Cay(\Gamma,S)$ it follows that $\Root(\gamma\cdot h)=h(\gamma^{-1})$ and $\Root(h)=h(1)$ are adjacent in $\mathcal{H}$ as $h$ is a homomorphism.
		Consequently, the map $\Root$ is a Borel homomorphism from $\HOM(\Gamma,S,\fH)$ to $\fH$ and $h\not=\gamma\cdot h$ as there are no loops in $\fH$.
	\end{proof}

	\paragraph{The $\Delta$-regular tree $T_\Delta$.}
	In this paper we only consider the case of the group 
	$$\mathbb{Z}^{*\Delta}_2=\langle \alpha_1,\dots,\alpha_\Delta | \alpha^2_1=\dots=\alpha^2_\Delta=1\rangle$$
	together with the generating set $S_\Delta=\{\alpha_1,\dots,\alpha_\Delta\}$. Since $\Cay(\mathbb{Z}^{*\Delta}_2,S_\Delta)$ is isomorphic to the $\Delta$-regular infinite tree, $T_\Delta$,  we use $\HOM(T_\Delta,\mathcal{H})$ to denote the graph $\HOM(\mathbb{Z}^{*\Delta}_2,S_\Delta,\mathcal{H})$. Note also that we consider $\Cay(\mathbb{Z}^{*\Delta}_2,S_\Delta)$ and $T_\Delta$ equipped with a $\Delta$-edge coloring.
	As suggested above, an equivalent description of the vertex set of $\HOM(T_\Delta,\mathcal{H})$ is that it is the set of pairs $(x,h)$ where $h$ is a homomorphism from the tree $T_\Delta$ to $\fH$ and $x$ is a distinguished vertex of $T_\Delta$, a root.
	Then we have that $(x,h)$ and $(y,g)$ form an ($\alpha$-)edge if and only if $h=g$ and $(x, y)$ is an ($\alpha$-)edge in $T_\Delta$.
	This is because (a) there is a one-to-one correspondence between a homomorphism from $\Cay(\mathbb{Z}^{*\Delta}_2,S_\Delta)$ to $\fH$ and the pairs $(x,h)$, and (b) the shift action $\mathbb{Z}^{*\Delta}_2 \curvearrowright \HOM(\mathbb{Z}^{*\Delta}_2,S_\Delta,\mathcal{H})$ corresponds to changing the root for a fixed homomorphism $h$ from $T_\Delta$ to $\fH$.

	Recall that an action $\Gamma \curvearrowright X$ is \emph{free} if for each $1 \neq \gamma \in \Gamma$ and $x \in X$ we have $\gamma \cdot x \neq x$. The \emph{free part}, denoted by $Free(X)$, is the set $\{x:\forall 1 \neq \gamma \in \Gamma \ \gamma \cdot x \neq x\}$. Note that the left-shift action of $\mathbb{Z}^{*\Delta}_2$ on, say, $\N^{\mathbb{Z}^{*\Delta}_2}$ is not free, in particular, the corresponding Schreier-graph has cycles. To remedy this, Marks used a restriction of the graph to the free part, showing that for each $\Delta > 2$ this graph has Borel chromatic number $\Delta+1$. Analogously, we have the following.
	\begin{definition}
		Let $\Homac(T_\Delta,\mathcal{H})=\HOM(T_\Delta,\mathcal{H}) \restriction Free(V(\mathcal{H})^{\mathbb{Z}^{*\Delta}_2}),$
		that is, the restriction of the graph $\HOM(T_\Delta,\mathcal{H})$ to the free part of the $\mathbb{Z}^{*\Delta}_2$ action.
	\end{definition}
	In our first application we will use this remedy to get acyclic graphs. 
	Note that for each edge in $\Homac(T_\Delta,\mathcal{H})$ there is a unique generator $\alpha\in S_\Delta$ that induces it.
	In particular, the graph $\Homac(T_\Delta,\mathcal{H})$ admits a canonical Borel edge $\Delta$-coloring.
	The following is straightforward.
	
	\begin{proposition}
		Let $\fH$ be a locally countable Borel graph.
		If $\Homac(T_\Delta,\fH)$ is nonempty, then it is $\Delta$-regular and acylic. 
	\end{proposition}
	
	However, utilizing the homomorphism graph together with an appropriate target graph, we will be able to completely avoid the non-free part, in an automatic manner. This way we will be able to guarantee the hyperfiniteness of the homomorphism graph as well. Recall that $T_\Delta=\Cay(\mathbb{Z}^{*\Delta}_2,S_\Delta)$ comes with a $\Delta$-edge coloring by the elements of $S_\Delta$. Let us consider the a subgraph of the homomorphism graph that arises by requiring $h$ to preserve this information.
	\begin{definition}	
		Assume that the graph $\mathcal{H}$ is equipped with a Borel edge $S_\Delta$-labeling. Let $\Homed(T_\Delta,\fH)$ be the restriction of $\HOM(T_\Delta,\fH)$ to the set
		\[\{h \in V(\HOM(T_\Delta,\fH)): \text{$h$ preserves the edge labels}\}.\]
	\end{definition}
	Clearly, $\Homed(T_\Delta,\fH)$ is also a Borel graph. Note that in the following statement the labeling of the edges of the target graph $\fH$ is typically not a coloring. 
	
	\begin{proposition}
		\label{pr:hypfin} Assume that $\mathcal{H}$ is an acyclic graph equipped with a Borel edge $S_{\Delta}$-labeling and $\Homed(T_\Delta,\mathcal{H})$ is nonempty. Then
		\begin{enumerate}
			\item \label{c:acyclic} $\Homed(T_\Delta,\mathcal{H})$ is acyclic,
			\item \label{c:hypfin} If $\mathcal{H}$ is hyperfinite, then so is $\Homed(T_\Delta,\mathcal{H})$,
			\item \label{c:regular} $\Homed(T_\Delta,\mathcal{H})$ is $\Delta$-regular.
		\end{enumerate}
	\end{proposition}

	\begin{proof}
		Observe that if $h$ is a homomorphism from a tree to an acyclic graph that is not injective, then there must be adjacent pairs of vertices $(x,y)$ and $(y,z)$ with $x \neq z$ and $h(z)=h(x)$. Thus, if $h \in \Homed(T_\Delta,\mathcal{H})$ is edge label preserving then it must be injective, as incident edges have different labels in $T_\Delta$. Therefore, the map $\Root$ is injective on each connected component of $\Homed(T_\Delta,\mathcal{H})$, yielding \eqref{c:acyclic}. 
		
		To see \eqref{c:hypfin}, let $(\mathcal{H}_n)_{n \in \N}$ be a witness to the hyperfiniteness of $\mathcal{H}$. Let $\mathcal{H}'_n$ be the pullback of $\mathcal{H}_n$ by the map $\Root$. Since $\Root$ is injective on every connected component, the graphs $\mathcal{H}'_n$ also have finite components and their union is $\Homed(T_\Delta,\mathcal{H})$.
		
		For \eqref{c:regular} just notice that using the injectivity of $\Root$ again, it follows that $\{\gamma \cdot h\}_{\gamma\in \{1\}\cup S_\Delta}$ has cardinality $\Delta+1$. 
	\end{proof}

	\section{Variations on Marks' technique}
	\label{s:marks}
	
	Now we are ready to adapt Marks' technique \cite{DetMarks} to homomorphism graphs. Let us denote by $\Deltachrom(\mathcal{H})$ the $\Game\mathbf{\Delta}^1_1$-chromatic number of $\mathcal{H}$ (see \cref{s:prel}). 
	
	\begin{customthm}{\ref{t:implication}} Let $\mathcal{H}$ be a locally countable Borel graph. Then \[\Deltachrom(\mathcal{H})>\Delta \implies \chi_B(\Homac(T_\Delta,\mathcal{H}))>\Delta.\]
	\end{customthm}

	The games we will define naturally yield elements $h \in \HOM(T_\Delta,\mathcal{H})$ rather than $\Homac(T_\Delta,\mathcal{H})$. In order to deal with the cyclic part of the graph, we will show slightly more, using the same strategy as Marks.
	Let $V \subseteq V(\HOM(T_\Delta,\mathcal{H}))$, an \emph{anti-game labeling of $V$} is a map $c:V \to \Delta$ such that there are no $i \in \Delta$ and distinct vertices $h,h'\in V$ with $c(h)=c(h')=i$ and $\alpha_i \cdot h=h'$.
	Observe that every $\Delta$-coloring is automatically anti-game labeling. 
	
	\begin{remark}\label{rem:anti-game}
		One can define analogously anti-game labelings for graphs with edges labeled by $\Delta$.
		Note that in the case when the graph is $\Delta$-regular and the labeling is an edge  $\Delta$-coloring, the existence of an anti-game labeling is equivalent with solving the well known \emph{edge grabbing problem} (that is, every vertex picks one adjacent edge but no edge can be picked from both sides, see \cite{definitionEdgeGrabbing}) or \emph{sinkless orientation problem}.
		Observe that the sinkless orientation problem, and hence the existence of anti-game labeling, can be easily solved on graphs with cycles.
		This observation is crucially utilized in both Marks' and this paper, see \cref{l:cyclic}.       
	\end{remark}

	\begin{lemma}
		\label{l:cyclic} There exists a Borel anti-game labeling $c:V(\HOM(T_\Delta,\mathcal{H})) \setminus V(\Homac(T_\Delta,\mathcal{H})) \to \Delta$.
	\end{lemma}
	\begin{proof}
		Let us use the notation $C=V(\HOM(T_\Delta,\mathcal{H})) \setminus V(\Homac(T_\Delta,\mathcal{H}))$. By definition, the $\mathbb{Z}^{*\Delta}_2$ action on every connected component $\HOM(T_\Delta,\mathcal{H}) \restriction C$ is not free. Using \cite[Lemma 7.3]{topics} we can find a Borel maximal family $\mathcal{F} \subseteq C^{<\N}$ of pairwise disjoint finite sequences each of length at least $2$ such that for each $(h_i)_{i<k} \in \mathcal{F}$ there is a sequence $(\alpha_{n_i})_{i<k} \in S^k_{\Delta}$ such that $\alpha_{n_i}\not=\alpha_{n_{i+1}}$, $\alpha_{n_i} \cdot h_i=h_{i+1}$ for $i<k-1$, and $\alpha_{n_0}\not=\alpha_{n_{k-1}}$, $\alpha_{n_{k-1}} \cdot h_{k-1}=h_0$. (Note that it is possible that $k=2$ in which case there are two distinct generators $\alpha_{n_0}\not= \alpha_{n_1}$ such that $\alpha_{n_0}\cdot h_0=\alpha_{n_1}\cdot h_0=h_1$.)
		
		Now label an element $h \in C$ by $n_i$ if $h=h_i$ for some $(h_i)_{i<k} \in \mathcal{F}$. Otherwise,  let $c(h)$ be the minimal $i$ such that $\alpha_i \cdot h$ has strictly smaller distance to $\mathcal{F}$ than $h$ with respect to the graph distance in $\HOM(T_\Delta,\mathcal{H})$. It is easy to check that $c$ is an anti-game labeling.
	\end{proof}

	\begin{proof}[Proof of \cref{t:implication}]
		We show that there is no Borel anti-game labeling $c:V(\HOM(T_\Delta,\mathcal{H})) \to \Delta$.
		Once we have that the proof of \cref{t:implication} is finished as follows.
		Suppose that $d$ is a Borel $\Delta$-coloring of $\Homac(T_\Delta,\mathcal{H})$.
		As observed above, every $\Delta$-coloring is also anti-game labeling.
		Consequently, the union of $d$ and the anti-game labeling produced in \cref{l:cyclic} is a Borel anti-game labeling of $V(\HOM(T_\Delta,\mathcal{H}))$, contradiction.
		
		Assume towards contradiction that $c:V(\HOM(T_\Delta,\mathcal{H})) \to \Delta$ is a Borel anti-game labeling.
		Without loss of generality we may assume that $\mathcal{H}$ has no isolated points. This ensures that the games below can be always continued. 
		
		We define a family of two-player games $\mathbb{G}(x,i)$ parametrized by elements $x \in V(\fH)$ and $i \in \Delta$.
		In a run of the game $\mathbb{G}(x,i)$ players $\PI$ and $\PK$ alternate and build a homomorphism $h$ from $T_\Delta$ to $\fH$, i.e., an element of $\HOM(T_\Delta,\fH) \subset V(\mathcal{H})^{\mathbb{Z}^{*\Delta}_2}$, with the property that $\Root(h)=x$.
		\begin{figure}

			\tikzstyle{mybox} = [draw=black, fill=blue!20, very thick,
			rectangle, rounded corners, inner sep=10pt]

			\tikzstyle{fancytitle1} =[fill=pastelred,rounded corners]
			\tikzstyle{fancytitle2} =[fill=pastelblue,rounded corners]
			\tikzstyle{legend} =[fill=pastelgreen,rounded corners]
			
			\begin{tikzpicture}
			\coordinate (0) at (0,0);
			
			\coordinate (a) at ($(0) +(0:1)$);
			\coordinate (b) at ($(0) +(140:1)$);
			\coordinate (c) at ($(0) +(220:1)$);
			
			\coordinate (ab) at ($(0) +(20:2)$);
			\coordinate (ac) at ($(0) +(-20:2)$);
			
			\coordinate (aba) at ($(0) +(30:3)$);
			\coordinate (abc) at ($(0) +(10:3)$);
			
			\coordinate (aca) at ($(0) +(-10:3)$);
			\coordinate (acb) at ($(0) +(-30:3)$);
			
			\coordinate (ba) at ($(0) +(130:2)$);
			\coordinate (bc) at ($(0) +(160:2)$);
			
			\coordinate (ca) at ($(0) +(200:2)$);
			\coordinate (cb) at ($(0) +(230:2)$);
			
			\coordinate (bab) at ($(0) +(140:3)$);
			\coordinate (bac) at ($(0) +(120:3)$);
			
			\coordinate (bcb) at ($(0) +(150:3)$);
			\coordinate (bca) at ($(0) +(170:3)$);
			
			\coordinate (cab) at ($(0) +(190:3)$);
			\coordinate (cac) at ($(0) +(210:3)$);
			
			\coordinate (cba) at ($(0) +(220:3)$);
			\coordinate (cbc) at ($(0) +(240:3)$);
			
			\coordinate (1st) at ($(0) +(270:1.5)$);
			\coordinate (2nd) at ($(0) +(270:2.5)$);
			\coordinate (edgelabel) at (0.5,0);
			
			\draw (0) -- (a);
			\draw (0) -- (b);
			\draw (0) -- (c);
			
			\draw (a) -- (ab);
			\draw (a) -- (ac);
			
			\draw (b) -- (ba);
			\draw (b) -- (bc);
			
			\draw (c) -- (ca);
			\draw (c) -- (cb);
			
			\draw (ab) -- (abc);
			\draw (ab) -- (aba);
			
			\draw (ac) -- (acb);
			\draw (ac) -- (aca);
			
			\draw (ac) -- (acb);
			\draw (ac) -- (aca);
			
			\draw (ba) -- (bac);
			\draw (ba) -- (bab);
			
			\draw (bc) -- (bca);
			\draw (bc) -- (bcb);
			
			\draw (ca) -- (cab);
			\draw (ca) -- (cac);
			
			\draw (cb) -- (cba);
			\draw (cb) -- (cbc);
			
			
			
			\node[above] at (0) {$x$};
			\node[below] at (edgelabel) {$\alpha_i$};
			\node[] at (1st) {1st};
			\node[] at (2nd) {2nd};
			
			\fill[color=pastelyellow] (0) circle(2pt);
			
			\foreach \dot in {b,c,ba,bc,ca,cb,cab,cac,cba,cbc,bab,bac,bca,bcb} {
				\fill[color=pastelblue] (\dot) circle(2pt);
			}
			
			\foreach \dot in {a,ab,ac,aba,abc,aca,acb} {
				\fill[color=pastelred] (\dot) circle(2pt);
			}
			
			\draw[dotted] (40:2.5) arc (40:-40:2.5) ;
			
			\draw[dotted] (110:2.5) arc (110:250:2.5) ;
			
			\draw[dotted] (40:1.5) arc (40:-40:1.5) ;
			
			\draw[dotted] (110:1.5) arc (110:250:1.5) ;
			
			\node [legend] (ll) at (5.5,0) {
				\begin{tikzpicture}[strong/.style={line width=0.5mm},weak/.style={line width=0.1mm}]
				\node[fancytitle1] at (0.5,1.2) {};
				\node[fancytitle2] at (0.5,0.6) {};
				
				\node[] (v) at (2,1.2)  {Alice labels};
				\node[] (j) at (2,0.6)  {Bob labels};
				
				\end{tikzpicture}
			};
			
		\end{tikzpicture}
		\centering
		
		\caption{The game $\mathbb{G}(x,i)$}
		\label{fig:game}
	\end{figure}
	
	In the $k$-th round, first $\PI$ labels vertices of distance $k$ from the $1$ on the side of the $\alpha_i$ edge.
	After that, $\PK$ labels all remaining vertices of distance $k$, etc (see  \cref{fig:game}).
	In other words, $\PI$ labels the elements of $\mathbb{Z}^{*\Delta}_2$ corresponding to reduced words of length $k$ starting with $\alpha_i$ then $\PK$ labels the rest of the reduced words of length $k$. 
	
	Note that once the parameters of the game are fixed the definition of the game is analoguous to the one defined by Marks \cite{DetMarks}.
	There is, however, one important difference.
	The allowed moves are vertices of the target graph $\fH$ restricted by its edge relation.
	Observe that the original construction of Marks can be interpreted in our setting by taking $\fH$ to be the the complete graph on $\mathbb{N}$.
	
	The winning condition is defined as follows: \[\text{$\PI$ wins the game $\mathbb{G}(x,i)$ iff $c(h) \neq i$}.\]

	\begin{lemma}\label{l:measurability}
		\begin{enumerate}
			\item \label{c:strategy} For any $x \in V(\mathcal{H})$ and $i \in \Delta$ one of the players has a winning strategy in the game $\mathbb{G}(x,i)$.  
			\item \label{c:delta12} The set $\{(x,i): \PI \text{ has a winning strategy in } \mathbb{G}(x,i)\}$ is $\Game \mathbf{\Delta}^1_1$.
		\end{enumerate}
	\end{lemma}
	\begin{proof}
		We encode the games $\mathbb{G}(x,i)$ in a way that allows the use of Borel determinacy theorem and \cref{l:gamesection}.
		
		Let us denote by $E_\mathcal{H}$ the connected component equivalence relation of $\mathcal{H}$. Observe that as $T_\Delta$ is connected, the range of any element $h \in \HOM(T_\Delta,\mathcal{H})$ is contained in a single $E_\mathcal{H}$ class. By the Feldman-Moore theorem, there is a countable collection of Borel functions $f_i:V(\mathcal{H}) \to V(\mathcal{H})$ such that $E_\mathcal{H}=\bigcup_{j \in \mathbb{N}} graph(f^{\pm 1}_j)$. Therefore, the games $\mathbb{G}(x,i)$ above can be identified by games played on $\mathbb{N}$, namely, labeling a vertex in $T_\Delta$ by a vertex $y\in V(\mathcal{H})$ corresponds to playing the minimal natural number $j$ with $f_j(x)=y$. Since the functions $(f_j)_{j \in \mathbb{N}}$ are Borel, this correspondence is Borel as well. Moreover, the rule that $h$ must be homomorphism determines a pruned subtree of legal positions $T_{x,i} \subset \mathbb{N}^{<\mathbb{N}}$ and the map $(x,i) \mapsto T_{x,i}$ is Borel. This yields that there exists a Borel set $B \subseteq V(\mathcal{H}) \times \Delta \times \N^\N$ such that
		\[\PI \text{ has a strategy in $\mathbb{G}(x,i)$} \iff \PI \text{ has a winning strategy in $G(T_{x,i},B_{x,i})$}.\]
		Now, the first claim follows from the Borel determinacy theorem, while the second follows from \cref{l:gamesection}.
	\end{proof}
	
	\begin{claim}\label{cl:PI}
		For every $x\in X$ there is an $i\in \Delta$ such that $\PK$ wins $\mathbb{G}(x,i)$.
	\end{claim}
	\begin{proof}
		Suppose not. Then we can combine strategies of $\PI$ for each $i$ in the natural way to build a homomorphism that is not in the domain of $c$ (see, e.g., \cite{Marks_Coloring} or \cite{DetMarks}).
	\end{proof}
	Now we can finish the proof of \cref{t:implication}. Define $d:V(\mathcal{H}) \to \Delta$ by 
	\begin{equation}d(x)=i \iff i \text{ is minimal such that $\PK$ has a winning strategy in $\mathbb{G}(x,i)$}.
	\label{e:coloring}
	\end{equation}
	Since $\Game\mathbf{\Delta}^1_1$ sets form an algebra, $d$ is $\Game\mathbf{\Delta}^1_1$ measurable and by \cref{cl:PI} it is everywhere defined. By our assumptions on $\mathcal{H}$ there are $x \neq x'$ adjacent with $d(x)=d(x')=i$. Now, we can play the two winning strategies corresponding to games $\mathbb{G}(x,i)$ and $\mathbb{G}(x',i)$ of $\PK$ against each other, as if the first move of $\PI$ was $x'$ (resp. $x$). This yields distinct homomorphisms $h,h'$ with $\alpha_i \cdot h=h'$ and $c(h)=c(h')=i$, contradicting that $c$ is an anti-game labeling.
\end{proof}

\subsection{Generalizations}
\label{ss:gen}

\paragraph{Edge labeled graphs.}
As mentioned above, a novel feature of our approach is that requiring the homomorphisms to be edge label preserving and ensuring that $\mathcal{H}$ is acyclic, we can get rid of the investigation of the cyclic part (see \cref{pr:hypfin}). In order to achieve this, we have to assume slightly more about the chromatic properties of the target graph. 

Assume that $\mathcal{H}$ is equipped with an edge $S$-labeling. The \emph{edge-labeled chromatic number, $\Edgelabel{\mathcal{H}}$} of $\mathcal{H}$ is the minimal $n$, for which there exists a map $c:V(\mathcal{H}) \to n$ so that for each $i \in n$ the set $c^{-1}(i)$ doesn't span edges with every possible label. 
In other words, $\Edgelabel{\mathcal{H}}>n$ if and only if no matter how we assign $n$ many colors to the vertices of $\mathcal{H}$, there will be a color class containing edges with every label. We define $\mu$ measurable, Baire measurable, etc. versions of the edge-labeled chromatic number in the natural way. 

\begin{theorem}
	\label{t:forg0} Let $\mathcal{H}$ be a locally countable Borel graph with a Borel $S_\Delta$-edge labeling, such that for every vertex $x$ and every label $\alpha$ there is an $\alpha$-labeled edge incident to $x$. Then \[\Edgelabeldelta{\mathcal{H}}>\Delta \implies \chi_B(\Homed(T_\Delta,\mathcal{H}))>\Delta.\]
\end{theorem}

\begin{proof}
	The proof is similar to the proof of \cref{t:implication}, but with taking the edge colors into consideration. 
	Let us indicate the required modifications. We define $\mathbb{G}(x,i)$ as above, with the extra assumption that players must build a homomorphism that respects edge labels, i.e., an element $h \in \Homed(T_\Delta,\mathcal{H})$. The condition on the edge-labeling ensures that the players can continue the game respecting the rules at every given finite step. 
	
	The analogue of \cref{cl:PI} clearly holds in this case, and we can define $d$ as in \eqref{e:coloring}. Finally, $\Edgelabeldelta{\mathcal{H}}>\Delta$ guarantees the existence of $i \in \Delta$ and $x,x' \in V(\mathcal{H})$ such that $d(x)=d(x')=i$ and that the edge between $x$ and $x'$ has label $\alpha_i$, which in turn allows us to use the winning strategies of $\PK$ in $\mathbb{G}(x,i)$ and $\mathbb{G}(x',i)$ against each other, as above. 
\end{proof}

\paragraph{Graph homomorphism.}
In what follows, we will consider a slightly more general context, namely, instead of the question of the existence of Borel colorings, we will investigate the existence of Borel homomorphisms to a given finite graph $H$. The following notion is going to be our key technical tool.

\begin{definition}[Almost $\Delta$-colorable]
	\label{def:DeltaStar}
	Let $\Delta>2$ and $H$ be a finite graph.
	We say that $H$ is almost $\Delta$-colorable if there are sets $R_0,R_1\subseteq V(H)$ such that $H$ restricted to $V(H) \setminus R_i$ has chromatic number at most $(\Delta-1)$ for $i\in \{0,1\}$, and there is no edge between vertices of $R_0$ and $R_1$.
\end{definition}

Note that if $\chi(H) \leq \Delta$, then $H$ is almost $\Delta$-colorable.
Indeed, if $A_1,\dots,A_\Delta$ are independent sets that cover $V(H)$, we can set $R_0=R_1=A_1$.
The basic properties of almost $\Delta$-colorable graphs are summarized in \cref{subsec:measure}.
In particular, we show that every almost $\Delta$-colorable graph has chromatic number at most $2\Delta-2$, a bound that appears in \cref{thm:MainBorel}.

\begin{theorem}
	\label{t:formeasure} Let $\mathcal{H}$ be a locally countable Borel graph and assume that $H$ is a finite graph that is almost $\Delta$-colorable.
	Assume that $\mathcal{H}$ is equipped with a Borel $S_\Delta$-edge labeling, with the property that for every vertex $v$ and every edge label $\alpha_i \in S_\Delta$ there exists and edge from $v$ with label $\alpha_i$. Then
	\[\Edgelabeldelta{\mathcal{H}}>2^{\Delta \cdot 2^{|V(H)|}} \implies \Homed(T_\Delta,\mathcal{H}) \text{ has no Borel homomorphism to $H$}.\]
\end{theorem}

\begin{proof}
	Assume for contradiction that such a Borel homomorphism $c$ exists. We will need a further modification of Marks' games. Let $R \subseteq V(H)$. For $x \in V(\mathcal{H})$ define the game $\mathbb{G}(x,i,R)$ as in the proof of \cref{t:forg0}, with the winning condition modified to 
	\[\text{$\PI$ wins the game $\mathbb{G}(x,i,R)$ iff $c(h) \not \in R$}.\]
	Observe that playing the strategies of $\PI$ against each other as in \cref{cl:PI} we can establish the following.
	
	\begin{claim}
		\label{cl:PIS} For every $x \in V(\mathcal{H})$ and every sequence $(R_{i})_{i\in\Delta}$ with $\bigcup_{i} R_{i}=V(H)$ there is some $i$ such that $\PI$ has no winning strategy in $\mathbb{G}(x,i,R_{i})$.  
	\end{claim}
	
	Now let $N$ be the powerset of the set $\{(i,R):i \in\Delta,R \subseteq V(H)\}$. Of course, $|N|=2^{\Delta \cdot  2^{|V(H)|}}$.  Define a mapping $d:V(\mathcal{H}) \to N$ by
	\[(i,R) \in d(x) \iff \text{$\PI$ has a winning strategy in $\mathbb{G}(x,i,R)$}.\]
	Using \cref{l:gamesection} as in the proof of \cref{l:measurability}, the map $d$ is $\Game\mathbf{\Delta}^1_1$-measurable. By our assumption on $\mathcal{H}$, there is a subset $C$ on which $d$ is constant and $C$ spans an edge with each label.
	\begin{lemma}\label{l:Borel(A)satisfied}
		Let $i\in \Delta$ and $R_0,R_1\subseteq V(H)$ be sets such that there is no edge between points of $R_0$ and $R_1$ in $G$.
		Then for every $x \in C$ $\PK$ has no winning strategy in at least one of $\mathbb{G}(x,i,R_0)$ and $\mathbb{G}(x,i,R_1)$. In particular, if $R$ is independent in $G$ then $\PK$ cannot have a winning strategy in $\mathbb{G}(x,i,R)$. 
	\end{lemma}
	\begin{proof}
		If there exists an $x\in C$ for which $\mathbb{G}(x,i,R_0)$ and $\mathbb{G}(x,i,R_1)$ can be won by $\PK$, then, as $d$ is constant on $C$, this is the case for every $x\in C$. So we could find $x_0,x_1 \in C$ connected with an $\alpha_i$ labeled edge so that $\PK$ has winning strategies in $\mathbb{G}(x_0,i,R_0)$ and $\mathbb{G}(x_1,i,R_1)$. Then we can  play the two winning strategies of $\PK$ against each other as in the proof of \cref{t:implication}. This would yield elements $h_0,h_1$ in $\Homed(T_\Delta,\mathcal{H})$ that form an $\alpha_i$-edge with $c(h_i)\in R_i$, contradicting our assumption on $c$ and $R_i$. 
	\end{proof}
	
	To finish the proof of the theorem, fix the sets $R_0,R_1$ from \cref{def:DeltaStar}, and take an arbitrary $x \in C$. 
	By \cref{l:Borel(A)satisfied}, we get that for one of them, say $R_0$, $\PI$ has a winning strategy $\mathbb{G}(x,\alpha_0,R_0)$.
	Let $A_1,\dots,A_{\Delta-1}$ be independent sets as in \cref{def:DeltaStar}, i.e., with the property that $R_0 \cup \bigcup_i A_i=V(G)$. Using \cref{l:Borel(A)satisfied} again, we obtain that $\PI$ has a winning strategy in $\mathbb{G}(x,\alpha_i,A_i)$ for each $i \in\Delta$. This contradicts \cref{cl:PIS}.  
	
\end{proof}

\section{Applications}
\label{s:applications}

In this section we apply the theorems proven before to establish our main results. We will choose a target graph using two prominent notions from descriptive set theory: \emph{category} and \emph{Ramsey property}.

\subsection{Complexity of the coloring problem}

First we will utilize the shift-graph $\mathcal{G}_S$ on $[\mathbb{N}]^\N$ to establish the complexity results. Let us mention that it would be ideal to use the main result of \cite{todorvcevic2021complexity} (i.e., that deciding the Borel chromatic number of graphs is complicated) directly and apply the $\Homac(T_\Delta,\cdot)$ map together with \cref{t:implication} to show that this already holds for acyclic bounded degree graphs. Unfortunately, since the mentioned theorem requires large $\Game\mathbf{\Delta}^1_1$-measurable chromatic number, this does not seem to be possible (the graphs constructed in \cite{todorvcevic2021complexity} only have large Borel chromatic numbers, at least a priori). Instead, we will rely on the uniformization technique from \cite{todorvcevic2021complexity}.
Roughly speaking, the technique enables us to prove that in certain situations deciding the existence of, say, Borel colorings is $\mathbf{\Sigma}^1_2$-hard, whenever we are allowed to put graphs ``next to each other".

Let $X,Y$ be uncountable Polish spaces, $\bbg$ be a class of Borel sets and $\Phi:\bbg(X) \to \bp(Y)$ be a map.  Define $\mc{F}^{\Phi}\subset \bbg(X)$ by $A \in \mc{F}^{\Phi} \iff \Phi(A) \not = \emptyset$ 
and let the \emph{uniform family, $\mc{U}^{\Phi}$,} be defined as follows: for $B \in \bbg(\om^\om \times X)$ let
\[\bar{\Phi}(B)=\{(s,y) \in  \om^\om \times Y:y \in \Phi(B_s)\},\]
and 
\[B \in \mc{U}^{\Phi} \iff \bar{\Phi}(B) \text{ has a full Borel uniformization}\]
(that is, it contains the graph of a Borel function $\oom \to Y$).

Let $\mathbf{\Delta}$ be a family of subsets of Polish spaces. Recall that a subset $A$ of a Polish space $X$ is \emph{$\mathbf{\Delta}$-hard,} if for every $Y$ Polish and $B \in \mathbf{\Delta}(Y)$  there exists a continuous map $f:Y \to X$ with $f^{-1}(A)=B$. A set is \emph{$\mathbf{\Delta}$-complete} if it is $\mathbf{\Delta}$-hard and in $\mathbf{\Delta}$. 
A family $\mc{F}$ of subsets of a Polish space $X$ is said to be \emph{$\mathbf{\Delta}$-hard on $\bbg$}, if there exists a set $B \in \bbg(\om^\om \times X)$ so that the set $\{s \in \om^\om:B_s \in \mc{F}\}$ is $\mathbf{\Delta}$-hard. The next definition captures the central technical condition. 
\begin{definition}
	\label{d:nicely}
	The  family $\mathcal{F}^{\Phi}$ is said to be \emph{nicely $\bs$-hard on $\bbg$} if for every $A \in \bs(\om^\om)$  there exist sets $B \in \bbg(\om^\om \times X)$ and $D \in \bs(\om^\om \times Y)$ so that $D \subset \bar{\Phi}(B)$ and for all $s \in \om^\om$ we have 
	\[  s \in A \iff D_s \not = \emptyset \iff \Phi(B_s)\not= \emptyset \ ( \ \iff B_s \in \mathcal{F}^\Phi).\]
	
\end{definition}	

A map $\Phi: \mathbf{\Gamma}(X)\to \mathbf{\Pi}^1_1(Y)$ is called \emph{$\mathbf{\Pi}^1_1$ on $\mathbf{\Gamma}$} if for every Polish space $P$ and $A\in \mathbf{\Gamma}(P \times X)$ we
have $\{(s,y) \in P \times Y:y \in \Phi(A_s)\}\in \mathbf{\Pi}^1_1$. Now we have the following theorem.

\begin{theorem}[\cite{todorvcevic2021complexity}, Theorem 1.6]
	
	\label{t:maincomplex}
	Let $X,Y$ be uncountable Polish spaces, $\bbg$ be a class of subsets of Polish spaces which is closed under continuous preimages, finite unions and intersections and $\mathbf{\Pi}^0_1 \cup \mathbf{\Sigma}^0_1 \subset \bbg$.   Suppose that $\Phi:\bbg(X) \to \mathbf{\Pi}^1_1(Y)$ is $\bp$ on $\bbg$ and that $\mathcal{F}^\Phi$ is nicely $\bs$-hard on $\bbg$. 
	Then the family $\mathcal{U}^\Phi$ is $\mathbf{\Sigma}^1_2$-hard on $\bbg$.
\end{theorem}

Let us identify infinite subsets of $\N$ with their increasing enumeration. If $x,y \in [\mathbb{N}]^{\mathbb{N}}$ let us use the notation $y \leq^\infty x$ in the case the set $\{n:y(n) \leq x(n)\}$ is infinite and $y \leq^* x$ if it is co-finite. Set $\mathcal{D}=\{(x,y):y \leq^\infty x\}.$ It follows form the fact that $\mathcal{G}_S$ restricted to sets of the form $\mathcal{D}_x$ has a Borel $3$-coloring that the graphs $\Homac(T_\Delta,\mathcal{G}_S\restriction \mathcal{D}_x)$ admit a Borel $3$-coloring, uniformly in $x$:

\begin{lemma} 
	\label{l:nondom} There exists a Borel function $f_{dom}: \omm \to \oom$ so that for each $x\in \omm$ we have $f_{dom}(x)=\langle c_0,\dots,c_{\Delta-1}\rangle$ with $c_i \in \mb{BC}(\omm)$, $\mb{A}(\omm)_{c_i}$ are $\Homac(T_\Delta,\mathcal{G}_S)$-independent subsets of $V(\Homac(T_\Delta,\mathcal{G}_S))$ for every $i<\Delta$ and \[V(\Homac(T_\Delta,\mathcal{G}_S\restriction \mathcal{D}_x)) = \bigcup^{\Delta-1}_{i=0} \mb{A}(\omm)_{c_i}.\]   
\end{lemma}

\begin{proof}
	Note that it suffices to construct a Borel map $c:[\mathbb{N}]^\mathbb{N} \times \Homac(T_\Delta,\mathcal{G}_S) \to \Delta$ that is a coloring of the graph $\Homac(T_\Delta,\mathcal{G}_S \restriction \mathcal{D}_x)$ for each $x$: indeed, we can use \cref{f:prel} for $(B_i)_x=\{(x,h):c(x,h)=i\}$ to obtain Borel maps $f_i:\omm \to \oom$ so that for every $x \in \omm$ we have $\mb{A}(\omm)_{f_i(x)}=B_i$ and let $f_{dom}(x)=\langle f_0(x),\dots,f_{\Delta-1}(x)\rangle$.
	
	It has been established in \cite[Lemma 4.5]{todorvcevic2021complexity} (see also \cite{di2015basis}) that there exists a Borel map $c':\mathcal{D} \to 3$ such that for each $x$ it is a $3$-coloring of the graph $\mathcal{G}_S \restriction \mathcal{D}_x$. As the map $\Root:\Homac(T_\Delta,\mathcal{G}_S) \to \mathcal{G}_S$ is a Borel homomorphism by \cref{pr:chromatic}, it follows that the map $c(x,h):=c'(x, \Root(h))$ is the desired $\Delta$-coloring (in fact, $3$-coloring).   
\end{proof}
Let $\mathcal{H}$ be the graph on $\oom \times V(\Homac(T_\Delta,\mathcal{G}_S))$ defined by making $(x,h), (x',h')$ adjacent if $x=x'$ and $h$ is adjacent to $h'$ in $\Homac(T_\Delta,\mathcal{G}_S)$. Fixing a Polish topology on $V(\Homac(T_\Delta,\mathcal{G}_S))$ that is compatible with the Borel structure, we might assume that $V(\mathcal{H})$ is a Polish space. 

Putting together results proved in the previous sections, we get the following corollary.

\begin{corollary}
	\label{c:shiftchrom}
	The Borel chromatic number of $\Homac(T_\Delta,\fG_S)$ is $\Delta+1$.
\end{corollary}
\begin{proof}
	This follows from \cref{f:provably}, \cref{t:galvinprikry}, and \cref{t:implication}. 
\end{proof}

Now we are ready to prove the following.
\begin{proposition}
	\label{pr:forH}
	There exists a Borel set $C \subseteq \oom \times \oom \times V(\Homac(T_\Delta,\mathcal{G}_S))$ so that the set $\{s: \chi_B(\mathcal{H} \restriction C_s)\leq \Delta\}$ is $\mb{\Sigma}^1_2$-hard.
\end{proposition}

\begin{proof}
	We check the applicability of \cref{t:maincomplex}, with $X=V(\Homac(T_\Delta,\mathcal{G}_S))$, $Y=\oom$, $\bbg=\mathbf{\Delta}^1_1$ and  \[\Phi(A)=\{c:(\forall x,y \in A)\Big(c=\langle c_0,\dots,c_{\Delta-1}\rangle \text{, $c_i \in \mb{BC}(X)$}, x \in \bigcup_i  \mb{C}(X)_{c_i} \]\[ \text{and }(x,y) \in \Homac(T_\Delta,\mathcal{G}_S) \Rightarrow (\forall i)\big(\lnot(x,y \in \mb{A}(X)_{c_i})\big)\Big) \},\]
	in other words, $\Phi(A)$ contains the Borel codes of the Borel $\Delta$-colorings of $\Homac(T_\Delta,\mathcal{G}_S) \restriction A$. Let $A \subseteq \oom$ be analytic and take a closed set $F \subset \oom \times \omm$ so that $\proj_0(F)=A$. Let \[B'=\{(s,y):(\forall x \leq^* y)(x \not \in F_s)\}.\]
	
	\begin{lemma}
		\label{l:needed}  
		\begin{enumerate}
			\item \label{c:b'Borel} $B' \in \mathbf{\Pi}^0_2$.
			\item $\Phi$ is $\mathbf{\Pi}^1_1$ on $\mathbf{\Delta}^1_1$.
			
			\item \label{c:ufi} For any Borel set $C$ we have $C \in \mathcal{U}^\Phi$ if and only if $\chi_B(\mathcal{H} \restriction C)\leq\Delta$.  
		\end{enumerate}
	\end{lemma} 
	\begin{proof}
		The first statement has been proved in the stated form in \cite[Lemma 4.6]{todorvcevic2021complexity}. The second statements also follow from the straightforward modification of the argument presented in \cite[Lemma 4.6]{todorvcevic2021complexity}: 
		in fact, $\Phi$ is $\mathbf{
			\Pi}^1_1$ on $\mathbf{\Delta}^1_1$ if $\Homac(T_\Delta,\mathcal{G}_S)$ is replaced with any Borel graph, and the last statement works for every Borel graph on a product space, where edges only go between points in the same vertical section, in particular, for $\mathcal{H}$.
	\end{proof}
	Now define 
	\[B=\{(s,h):h \in \Homac(T_\Delta,\mathcal{G}_S \restriction B'_s)\},\]
	and \[D=\{(s,c):s \in A \text{ and }(\exists x \in F_s) (f_{dom} (x)=c)\},\] where $f_{dom}$ is the function from \cref{l:nondom}. 
	
	We will show that $B$ and $D$ witness that $\mathcal{F}^\Phi$ is nicely $\mathbf{\Sigma}^1_1$-hard. The set $B$ is Borel by \eqref{c:b'Borel} of the lemma above, while by its definition $D$ is analytic. 
	
	Suppose that $s \in A$. Then for each $x' \in F_s$ we have 
	$$B'_s= \{y:(\forall x \leq^* y)(x \not \in F_s)\} \subset \{y:y \leq^\infty x'\}=\fD_{x'}.$$
	Thus, by \cref{l:nondom} $B_s \in \mathcal{F}^\Phi$ and $D_s \not = \emptyset$. Moreover, if $c \in D_s$ then for some $x \in F_s$ we have $f_{dom}(x)=c$ with $c=\langle c_0,\dots,c_{\Delta-1}\rangle$, again by \cref{l:nondom} we have $B_s \subseteq \bigcup^{\Delta-1}_{i=0} \mb{A}(\omm)_{c_i}$  and the sets $\mb{A}(\omm)_{c_i}$ are $\Homac(T_\Delta,\mathcal{G}_S)$-independent, thus, $D_{s} \subseteq \Phi(B_s)$. Conversely, if $s \not \in A$ then $F_s=D_s=\emptyset$ and $B'_s=\omm$. Then $B_s=\Homac(T_\Delta,\mathcal{G}_S)$, which set does not admit a Borel $\Delta$-coloring by \cref{c:shiftchrom}. Consequently, $\Phi(B_s)=\emptyset$. 
	
	So, \cref{t:maincomplex} is applicable and it yields a Borel set $C \subseteq \oom \times \oom \times \omm$ so that $\{s:C_s \in \mathcal{U}^\Phi\}$ is $\mb{\Sigma}^1_2$-hard. This implies the desired conclusion by \eqref{c:ufi} of the Lemma above. 
\end{proof}

We can prove \cref{t:mainc}. Let us restate the theorem, describing precisely what we mean by ``form a $\mathbf{\Sigma}^1_2$-complete set".  

\begin{customthm}{\ref{t:mainc}}
	Let $X$ be an uncountable Polish space and $\Delta>2$.
	The set \[S=\{c\in \mathbf{BC}(X^2): \text{$\mathbf{C}(X^2)_c$ is a $\Delta$-regular acyclic Borel graph with Borel chromatic number $\leq \Delta$} \}\]
	is $\mathbf{\Sigma}^1_2$-complete. 
	
	In particular, Brooks' theorem has no analogue for Borel graphs in the following sense: there is no countable family $\{\mathcal{H}_i\}_{i \in I}$ of Borel graphs such that for any  Borel graph $\mathcal{G}$ with $\Delta(\mathcal{G})\leq \Delta$ we have $\chi_B(\mathcal{G})>\Delta$ if and only if for some $i \in I$ the graph $\mathcal{G}$ contains a Borel homomorphic copy of $\mathcal{H}_i$.
\end{customthm}
\begin{proof}[Proof of \cref{t:mainc}]
	
	First, note that using the fact that the codes of Borel functions between Polish spaces form a $\mathbf{\Pi}^1_1$ set, it is straightforward to show that $S$ is a $\mathbf{\Sigma}^1_2$ set (see e.g., \cite[Proof of Theorem 1.3]{todorvcevic2021complexity}). Similarly, one can check that if there was a collection $\{\mathcal{H}_i:i \in I\}$ as above, then this would yield that the set $S$ is $\mathbf{\Pi}^1_2$. Thus, in order to show both parts of the theorem it suffices to prove that $S$ is $\mathbf{\Sigma}^1_2$-hard.
	
	Second, by \cite{sabok}, it follows that if we replace continuous functions with Borel ones in the definition of $\mathbf{\Sigma}^1_2$-hard sets we get the same class. As uncountable Polish spaces are Borel isomorphic, it is enough to show that $S$ is $\mathbf{\Sigma}^1_2$-hard for some $X$.
	
	Take the graph $\mathcal{H}$ and the set $C$ from \cref{pr:forH}.  Note that the graph $\Homac(T_\Delta,\mathcal{G}_S)$ is acyclic and has degrees $\leq \Delta$ by its construction. Therefore, the same holds for $\mathcal{H} \restriction C_s$ for each $s$. Using that the sets $D_i=\{(s,x,h):\text{the degree of $(x,h)$ in $\mathcal{H} \restriction C_s$ is $i$}\}$ are Borel, it is straightforward to modify $\mathcal{H}$ so that we obtain a Borel graph $\mathcal{H}_\Delta$ on a Polish space of the form $X=\N^\N \times Y$ such that for each $s$ the graph $\mathcal{H}_\Delta \restriction \{s\} \times Y$ is $\Delta$-regular, acyclic and that the set $\{s:\chi_B(\mathcal{H}_\Delta \restriction \{s\} \times Y))
	\leq \Delta\}=\{s:\chi_B(\mathcal{H}\restriction C_s)\leq \Delta\}$ (indeed, to vertices in $D_i$ we can attach $\Delta-i$-many disjoint infinite rooted trees that are $\Delta$-regular except for the root, which has degree $\Delta-1$, in a Borel way). The third part of \cref{f:prel} gives a Borel reduction from the former set to $S$. Since the latter set is $\mathbf{\Sigma}^1_2$-hard, this yields the desired result by using \cite{sabok} as above.
\end{proof}

\subsection{Hyperfiniteness}

In this section we use Baire category arguments to obtain a new proof of \cref{t:hyperfiniteness}.
\begin{customthm}{\ref{t:hyperfiniteness}}
	There exists a hyperfinite $\Delta$-regular acyclic Borel graph with Borel chromatic number $\Delta+1$.  
\end{customthm}

We will utilize a version of the graph $\mathbb{G}_0$ constructed in \cite{KST}. For $s\in 2^{<\mathbb N}$ define 
$$\mathbb G_s=\{(s^\frown (0)^\frown c,s^\frown (1)^\frown c):c\in 2^\mathbb{N}\}$$
on $2^\mathbb N$.
Fix some collection $(s_n)_{n\in \mathbb N}\subseteq 2^{<\mathbb N}$ such that $|s_n|=n$, i.e., $s_n\in 2^n$, for every $n\in \mathbb N$, together with a function $e:\mathbb{N}\to \Delta$ such that $(s_n)_{e(n)=i}\subseteq 2^{<\mathbb N}$ is dense in $2^{<\mathbb{N}}$ for every $i\in \Delta$. Set $\mathbb{G}_0=\bigcup_{n \in \N}\mathbb{G}_{s_n}$. Label an edge $\alpha_i$ if it is in the graph $\bigcup_{e(n)=i} \mathbb{G}_{s_n}$. Finally, write $\fH$ for the restriction of $\mathbb{G}_0$ to those vertices $x$ such that every vertex in the connected component of $x$ is adjacent to at least one edge of each label. Standard arguments yield the following claim.
\begin{claim}
	\label{cl:Bairecat}
	\begin{enumerate}
		\item $\mathcal{H}$ is acyclic and locally countable.
		\item $\mathcal{H}$ is defined on a comeager subset of $2^\N$.
		\item The Baire measurable edge-labeled chromatic number of $\mathcal{H}$ is infinite (in fact, uncountable).
		\item $\mathcal{H}$ is hyperfinite.
	\end{enumerate}
\end{claim}
\begin{proof}
	The fact the $\mathcal{H}$ is locally countable is clear from its definition, while acyclicity follows from the assumption that $|s_n|=n$. To see the second part, note that the set $\{x \in 2^\mathbb{N}: \forall i \in \Delta \  \exists n \ (e(n)=i \land s_n \sqsubset x) \}$ is open and dense in $2^\N$. Now, $\mathcal{H}$ is the restriction of $\mathbb{G}_0$ to a set that is an intersection of the image of this set under countably many homeomorphisms of the form $s_n^\frown (i)^\frown c \mapsto s_n^\frown (1-i)^\frown c$, hence its vertex set is comeager. 
	
	The proof of the third part is identical to the proof of \cite[Proposition 6.2]{KST}.
	We include the argument for completeness.  
	Assume that $c:2^\N \to \aleph_0$ is a Baire measurable coloring and $j \in \Delta$ is arbitrary.
	Then, for some $i$, the set $c^{-1}(i)$ is non-meager.
	Since $c^{-1}(i)$ is Baire-measurable, there is some neighborhood $N_t$ such that $N_t \setminus c^{-1}(i)$ is comeager. In turn, there is some $n$ with $e(n)=j$ and $N_{s_n} \setminus c^{-1}(i)$ comeager.
	Since the map $s_n\concatt(i) \concatt r \mapsto s_n \concatt (1-i) \concatt r$ is category preserving from $N_{s_n}$ to $N_{s_n}$, there will be some $x,y \in c^{-1}(i) \cap N_{s_n}$, with $(x,y) \in \mathbb{G}_{s_n}$, in other words, the edge $(x,y)$ is labeled $j$.  
	
	Finally, the hyperfiniteness of $\mathcal{H}$ follows from the fact that $E_\mathcal{H} \subset \mathbb{E}_0$ (see, e.g., \cite[Proposition 1.3]{jackson2002countable}). \end{proof}
\begin{proof}[Proof of \cref{t:hyperfiniteness}]
	By \cref{cl:Bairecat} and \cref{pr:hypfin} the graph $\Homed(T_\Delta,\mathcal{H})$ is hyperfinite, $\Delta$-regular and acyclic. By \cref{t:forg0} and \cref{f:provably} its Borel chromatic number is $\Delta+1$.
\end{proof}

Note that the above theorem also implies Theorem~1.6 in \cite{conleyhyp}, namely, that there is no Borel version of the Lovász Local Lemma (LLL) even on hyperfinite graphs and if the probability of a bad event is polynomial in the degree of the dependency graph (for related results and precise definitions see \cite{OlegLLL} and \cite{Bernshteyn2021LLL}).  
In order to see this, observe that the sinkless orientation problem from \cite{brandt_etal2016LLL} can be thought of as an instance of the LLL as follows: Each edge corresponds to a random binary variable representing its orientation. At each node the bad event has probability $2^{-\Delta}$: all incident edges are oriented towards it.
It remains to observe that a Borel solution to the sinkless orientation problem implies easily a Borel solution to the edge grabbing problem which in turn, by \cref{rem:anti-game}, implies the existence of a Borel anti-game labeling.
However, it follows from the proof of \cref{t:forg0} that $\Homed(T_\Delta,\mathcal{H})$ does not admit a Borel anti-game labeling.

\subsection{Graph homomorphisms}
\label{subsec:measure}

In this section we prove \cref{thm:MainBorel} that we restate here for the convenience of the reader.

\MainBorel*

We remark that the first proof of the theorem (without the conclusion about hyperfiniteness) relied on a construction from the random graph theory, see \cite{brandt_chang_grebik_grunau_rozhon_vidnyaszky2021LCLs_on_trees_descriptive} for motivation and connection to the $\local$ model.
We sketch the construction here for completeness.
Fix $k\in \mathbb{N}$, large enough depending on $\Delta$, and consider $k\Delta$ pairings on a set $n$ sampled independently uniformly at random.
In other words, we have a $k\Delta$-regular graph and there is a canonical edge $\Delta$-labeling such that each vertex is adjacent to exactly $k$ edges of each color.
Now taking a \emph{local-global} limit of such graphs as $n\to \infty$ produces with probability $1$ an acyclic graphing with large edge-labeled chromatic number as needed, see \cite{bollobas,hatamilovaszszegedy}.

Before we prove \cref{thm:MainBorel}, we discuss graphs that are almost $\Delta$-colorable.

\paragraph{Almost $\Delta$-colorable graphs.}
As we have seen, in \cref{ss:gen}, being almost $\Delta$-colorable is (formally) a weaker condition than having chromatic number at most $\Delta$, but still allows us to use a version of Marks' technique. 
Similarly to the way the complete graph on $\Delta$-many vertices, $K_\Delta$, is maximal among graphs of chromatic number $\leq \Delta$, we show that there exists a maximal graph (under homomorphisms) that is almost $\Delta$-colorable. It turns out that the chromatic number of the maximal graph is $2\Delta-2$.

\begin{figure}
	\centering
	\includegraphics[width=.6\textwidth]{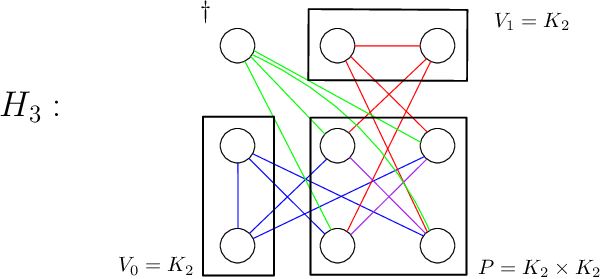}
	\caption{The maximal graph that is almost $\Delta$-colorable for $\Delta=3$.}
	\label{fig:h3}
\end{figure}

Let us describe the maximal examples of graphs that are almost $\Delta$-colorable for $\Delta>2$.
Recall that the (categorical) product $G \times H$ of graphs $G,H$ is the graph on $V(G) \times V(H)$, such that $((g,h),(g',h')) \in E(G \times H)$ if and only if $(g,g') \in E(G)$ and $(h,h') \in E(H)$.

Write $P$ for the product $K_{\Delta-1}\times K_{\Delta-1}$. Let $V_0$ and $V_1$ be vertex disjoint copies of $K_{\Delta-1}$.
We think of vertices in $V_i$ and $P$ as having labels from $[\Delta-1]$ and $[\Delta-1] \times [\Delta-1]$, respectively.
The graph $H_\Delta$ is the disjoint union of $V_0$, $V_1$, $P$ and an extra vertex $\dagger$ that is connected by an edge to every vertex in $P$, and additionally, if $v$ is a vertex in $V_0$ with label $i\in [\Delta-1]$, then we connect it by an edge with $(i',j)\in P$ for every $i'\not=i$ and $j\in [\Delta-1]$, and if $v$ is a vertex in $V_1$ with label $j\in \Delta-1$, then we connect it by an edge with $(i,j')\in P$ for every $j'\not=j$ and $i \in [\Delta-1]$. The graph $H_{3}$ is depicted in \cref{fig:h3}.

\begin{proposition}\label{pr:MinimalExample}
	\begin{enumerate}
		\item \label{prop:hdelta}$H_\Delta$ is almost $\Delta$-colorable.
		\item \label{prop:hdeltachrom} $\chi(H_\Delta)=2\Delta-2$.
		\item \label{prop:hdeltamax} A graph $G$ is almost $\Delta$ colorable if and only if it admits a homomorphism to $H_\Delta$.
	\end{enumerate}
	
\end{proposition}
\begin{proof}
	\eqref{prop:hdelta} Set $R_0=V(V_0)\cup \{\dagger\}$ and $R_1=V(V_1)\cup \{\dagger\}$.
	By the definition there are no edges between $R_0$ and $R_1$.
	Consider now, e.g., $V(H_\Delta)\setminus R_0$.
	Let $A_j$ consist of all elements in $P$ that have second coordinate equal to $j$ together with the vertex in $V_1$ that has the label $j$.
	By the definition, the set $A_i$ is independent and $\bigcup_{i\in [\Delta-1]}A_i$ covers $H_\Delta\setminus R_0$, and similarly for $R_1$.
	
	\eqref{prop:hdeltachrom} First we show that $\chi(H_\Delta)\le 2\Delta-2$.
	Observe that there is no edge between $R_0 \setminus R_1$ and $R_0 \cap R_1$, as there is no edge between $R_0$ and $R_1$.
	It follows that the chromatic number of the induced subgraph of $H_\Delta$ to $R_0$ is $\Delta-1$.
	The desired $2\Delta-2$ coloring of $H_\Delta$ is then defined as the disjoint union of the $\Delta-1$-colorings of $R_0$ and $V(H) \setminus R_0$.
	
	Next we show that $\chi(H_\Delta)\geq 2\Delta-2$.
	Towards a contradiction, assume that $c$ is a coloring of $H_\Delta$ with $<2\Delta-2$-many colors. It follows that  $|c(V(P))|\leq 2\Delta-4$, and also $\Delta-1 \leq |c(V(P))|$.
	
	First we claim that there are no indices $i,j\in [\Delta-1]$ (even with $i=j$) such that $c(i,r)\not=c(i,s)$ and $c(r,j)\not=c(s,j)$ for every $s\not=r$: indeed, otherwise, by the definition of $P$ we had $c(i,r)\not=c(s,j)$ for every $r,s$ unless $(i,r)=(s,j)$, which would the upper bound on the size of $c(V(P))$.
	
	Therefore, without loss of generality, we may assume that for every $i\in [\Delta-1]$ there is a color $\alpha_i$ and two indices $j_i \neq j'_i$ such that $c(i,j_i)=c(i,j'_i)=\alpha_i$. It follows form the definition of $P$ and $j_i \neq j'_i$ that
	$\alpha_i\not=\alpha_{i'}$ whenever $i\not= i'$. 
	
	Moreover, note that any vertex in $V_1$ is connected to at least one of the vertices $(i,j_i)$ and $(i,j'_i)$, hence none of the colors $\{\alpha_i\}_{i \in [\Delta-1]}$ can appear on $V_1$. Consequently, since $V_1$ is isomorphic to $K_{\Delta-1}$ we need to use at least $\Delta-1$ additional colors, a contradiction.

	\eqref{prop:hdeltamax} First note that if $G$ admits a homomorphism into $H_\Delta$, then the pullback of the sets $R_0$, $R_1$ and $A_i$ for $i\in [\Delta-1]$ witnesses that $G$ is almost $\Delta$-colorable.
	
	Conversely, let $G$ be almost $\Delta$-colorable.
	Fix the corresponding sets $R_0,R_1$ together with $(\Delta-1)$-colorings $c_0,c_1$ of their complements.
	We construct a homomorphism $\Theta$ from $G$ to $H_\Delta$. Let
	\[\Theta(v)=
	\begin{cases} \dagger &\mbox{if } v \in R_0 \cap R_1, \\
	c_0(v) & \mbox{if } v \in R_1 \setminus R_0, \\
	c_1(v) & \mbox{if } v \in R_0 \setminus R_1,
	\\
	(c_0(v),c_1(v)) & \mbox{if } v \not \in R_0 \cup R_1.
	\end{cases} \]
	
	Observe that $R=R_0\cap R_1$ is an independent set such that there is no edge between $R$ and $R_0 \cup R_1$.
	Using this observation, one easily checks case-by-case that $\Theta$ is indeed a homomorphism.
\end{proof}

\begin{remark}
	It can be shown that for $\Delta=3$ both the Chv\' atal and Gr\" otsch graphs are almost $\Delta$-colorable.
\end{remark}

\begin{proof}[Proof of \cref{thm:MainBorel}]
	Note that it is enough to show the existence of such $H$ and $\mathcal{G}$ with $\chi(H)=2\Delta-2$. Indeed, since erasing a vertex decreases the chromatic number by at most $1$, we can produce subgraphs of $H$ with chromatic number exactly $\ell$ for each $\ell \leq 2\Delta-2$.
	
	By \cref{pr:MinimalExample}, the graph $H_\Delta$ is almost $\Delta$-colorable and has chromatic number $2\Delta-2$.
	Then it is easy to see that taking the target graph $\fH$ as in \cref{cl:Bairecat} gives the conclusion by \cref{pr:hypfin} and \cref{t:formeasure}.
\end{proof}

\begin{remark}
	Interestingly, recent results connected to counterexamples to Hedetniemi's conjecture yield \cref{thm:MainBorel} asymptotically, as $\Delta \to \infty$. 
	\label{r:hedet}Recall that Hedetniemi's conjecture is the statement that if $G,H$ are finite graphs then $\chi(G \times H)=\min\{\chi(G),\chi(H)\}$. This conjecture has been recently disproven by Shitov \cite{shitov}, and strong counterexamples have been constructed later (see, \cite{tardif2019note,zhu2021note}). We claim that these imply for $\varepsilon>0$ the existence of finite graphs $H$ with $\chi(H) \geq (2-\varepsilon)\Delta$ to which $\Delta$-regular Borel forests cannot have, in general, a Borel homomorphism, for every large enough $\Delta$. Indeed, if a $\Delta$-regular Borel forest admitted a Borel homomorphism to each finite graph of chromatic number at least $(2-\varepsilon)\Delta$, it would have such a homomorphism to their product as well. Thus, we would obtain that the chromatic number of the product of any graphs of chromatic number $(2-\varepsilon)\Delta$ is at least $\Delta+1$. This contradicts Zhu's result \cite{zhu2021note}, which states that the chromatic number of the product of graphs with chromatic number $n$ can drop to $\approx \frac{n}{2}$.
\end{remark}

\section{Remarks and further directions}

\label{s:problems}

Since the construction of homomorphism graphs is rather flexible, we expect that this method will find further applications.
A direction that we do not take in this paper is to investigate homomorphism graphs corresponding to countable groups other than $\mathbb{Z}^{*\Delta}_2$. Another possible direction could be to understand the connection of our method with hyperfiniteness. 


While \cref{t:mainc} is optimal in the Borel context, one might hope that there is a positive answer in the case of graphs arising as compact, free subshifts of $2^{\mathbb{Z}^{*\Delta}_2}$. 

\begin{question}
	Is there a characterization of Borel graphs with Borel chromatic number $\leq \Delta$ that are compact, free subshifts of the left-shift action of $\mathbb{Z}^{*\Delta}_2$ on $2^{\mathbb{Z}^{*\Delta}_2}$? 
\end{question}
A way to answer this question on the negative would be to extend the machinery developed in \cite{seward2016borel} or \cite{Bernshteyn2021local=cont}, so that the produced equivariant maps preserve the Borel chromatic number, their range is compact, and then apply \cref{t:mainc}; however, this seems to require a significant amount of new ideas.

Let us point out that \cref{t:implication} has a particularly nice form, if we assume Projective Determinacy or replace the Axiom of Choice with the Axiom of Determinacy (see, e.g., \cite{trichotomy} for related results). 

\begin{theorem}
	Let $\Delta>2$.
	\begin{itemize}
		\item ($\PD$) Let $\mathcal{H}$ be a locally countable Borel graph. Then 
		\[\chi_{pr}(\mathcal{H})>\Delta \iff \chi_{pr}(\Homac(T_\Delta,\mathcal{H}))>\Delta,\]
		where $\chi_{pr}$ stands for the projective chromatic number of $\mathcal{H}$.
		\item ($\AD+\DC$) Let $\mathcal{H}$ be a locally countable graph on a Polish space. Then 
		\[\chi(\mathcal{H})>\Delta \iff \chi(\Homac(T_\Delta,\mathcal{H}))>\Delta.\]
	\end{itemize}
\end{theorem}

\subsection*{Acknowledgements} 
We would like to thank Anton Bernshteyn, Mohsen Ghaffari, Steve Jackson, Alex Kastner, Alexander Kechris, Yurii Khomskii, Clark Lyons, Andrew Marks, Oleg Pikhurko, Brandon Seward, Jukka Suomela, and Yufan Zheng for insightful discussions.

\bibliographystyle{alpha}
\bibliography{ref}

\end{document}